\documentclass{amsart}
\usepackage{graphicx} 
\usepackage{amsmath}
\usepackage{amssymb} 
\usepackage{amsfonts} 
\usepackage{amsthm}
\usepackage{float}
\usepackage{fdsymbol} 
\usepackage{makecell} 
\usepackage{url}

\title{Co-intersection graphs of nonabelian finite simple groups have diameter two}
\author{Henry Bradford}
\address{Christ’s College, University of Cambridge, St Andrew’s Street, Cambridge, England, CB2 3BU}
\email{hb470@cam.ac.uk}
\author{Kamilla Rekv\'enyi} 
\address{Department of Mathematics, University of Manchester, M13 9PL Manchester, UK. Also
affiliated with: Heilbronn Institute for Mathematical Research, BS8 1UG Bristol, UK.}
\email{kamilla.rekvenyi@manchester.ac.uk}
\date{\today}

\newtheorem{thm}{Theorem}[section]
\newtheorem{lem}[thm]{Lemma}
\newtheorem{propn}[thm]{Proposition}

\newtheorem{defn}[thm]{Definition}
\newtheorem{ex}[thm]{Example}

\newtheorem{conj}[thm]{Conjecture}

\newtheorem{rmrk}[thm]{Remark}
\newtheorem{qu}[thm]{Question}

\DeclareMathOperator{\ccl}{ccl}
\DeclareMathOperator{\Co}{Co}
\DeclareMathOperator{\coIG}{coIG}
\DeclareMathOperator{\diam}{diam}
\DeclareMathOperator{\GL}{GL}
\DeclareMathOperator{\GU}{GU}
\DeclareMathOperator{\IG}{IG}
\DeclareMathOperator{\SL}{SL}
\DeclareMathOperator{\SO}{SO}
\DeclareMathOperator{\Sp}{Sp}
\DeclareMathOperator{\PSL}{PSL}
\DeclareMathOperator{\PSp}{PSp}

\begin{document}

\begin{abstract}
Let $G$ be a finite group. 
The co-intersection graph $\Delta_G ^c$ of $G$ has as its vertices the 
nontrivial proper subgroups of $G$, with edges joining those pairs of 
subgroups which intersect trivially. 
It is clear that every connected component of $\Delta_G ^c$ 
has diameter at most three. 
In this Note, we show that when $T$ is a nonabelian finite simple group, 
$\Delta_T$ is connected of diameter exactly two. 
We also make several observations about the extent to which 
a finite group is determined by its co-intersection graph, 
and about the class of finite graphs arising as co-intersection graphs. 
\end{abstract}

\maketitle

\section{Introduction}

Throughout this paper, $T$ denotes a nonabelian finite simple group. 
For $G$ a finite group, the \emph{intersection graph} $\Delta_G$ of $G$ 
is the finite graph, the vertices of which are given by the nontrivial 
proper subgroups of $G$, and in which two subgroups are adjacent iff 
they intersect nontrivially \cite{CsaPol}. 
Graph-theoretic properties of the intersection graphs of finite groups 
have received substantial attention in the literature 
(see the references in the Introductions to \cite{DevRaj,VisVad}). 
In particular, much is now known about the diameters 
of intersection graphs (see \cite{LeeRek} and references therein). 
Less well-studied is the \emph{co-intersection graph:} 
the graph-theoretic complement $\Delta_G ^c$ of $\Delta_G$. 

For $\Gamma$ a finite connected graph, 
the \emph{diameter} $\diam(\Gamma)$ of $\Gamma$ 
is the greatest path-distance which occurs between pairs of vertices in $\Gamma$. 
Our main result is as follows. 

\begin{thm} \label{MainThmIntro}
For all $T$, $\Delta_T ^c$ is connected, with $\diam (\Delta_T ^c) = 2$. 
\end{thm}

For every finite group $G$, every connected component of $\Delta_G ^c$ 
has diameter at most three (see Lemma \ref{Diam3Lemma} below), 
and the groups for which $\Delta_G ^c$ is connected of diameter one 
(that is, a complete graph) were completely classified in \cite{VisVad} 
(see Proposition \ref{CompProp} below). 
One may exhibit many finite groups $G$ for which $\Delta_G ^c$ is 
connected of diameter exactly three: 
for instance the dihedral groups $D_{2^n}$ for $n \geq 3$ are as such. 
The case of Theorem \ref{MainThmIntro} for which $T$ is an alternating 
group is Proposition 2.17 of \cite{VisVad}. 

We say that a finite group $G$ has Property $(P)$ if, for any two maximal subgroups $M_1 , M_2 \leq G$, there exists an element $g \in G$ of prime order such that $g \notin M_1 \cup M_2$. 
Our key technical result is as follows. 

\begin{propn}
Every nonabelian finite simple group has Property $(P)$. 
\end{propn}

Theorem \ref{MainThmIntro} is then immediate from the classification of finite groups 
with complete co-intersection graph (discussed above) and the following Lemma. 

\begin{lem}
Let $G$ be a finite group. 
Suppose $G$ has at least two nontrivial proper subgroups. 
If $G$ has Property $(P)$ then $\Delta_G ^c$ is connected 
with $\diam (\Delta_G ^c) \leq 2$. 
\end{lem}

\begin{proof}
Suppose $G$ has Property $(P)$, and let $H_1 , H_2 \leq G$ be nontrivial proper subgroups. 
For $i=1,2$, let $M_i \leq G$ be a maximal subgroup containing $H_i$. 
Let $g \in G$ be as in Property $(P)$. 
Then $H_1 \text{---} \langle g \rangle \text{---} H_2$ is a path in $\Delta_G ^c$. 
\end{proof}

\begin{lem} \label{PrimeDivLem}
Let $p$ be a prime divisor of $\lvert T \rvert$ and suppose $M_1 , M_2 \leq T$ are maximal subgroups 
such that $p \nmid \lvert M_1 \rvert$. 
Then there is an element $g$ of order $p$ such that $g \notin M_1 \cup M_2$. 
\end{lem}

\begin{proof}
Let $X_p$ be the set of all elements of $T$ of order $p$, 
so that by simplicity of $T$, $\langle X_p \rangle = T$. 
Then there exists $g \in X_p$ such that $g \notin M_2$, 
and so $g \notin M_1 \cup M_2$. 
\end{proof} 

\begin{propn}
    Suppose $T$ does not contain a maximal subgroup $M$ such that $\lvert T \rvert$ and 
    $\lvert M \rvert$ have the same prime divisors. 
    Then $T$ has Property $(P)$. 
\end{propn}

\begin{proof}
    This is now immediate from Lemma \ref{PrimeDivLem}. 
\end{proof}

A list of all pairs $(T,M)$ such that $M$ is a maximal subgroup of $T$, with 
$\lvert T \rvert$ and $\lvert M \rvert$ having the same prime divisors, 
is provided by Burness and Covato \cite{BurCov}, 
following \cite{LiePraSax}. 
For a finite group $G$ and $n > 1$, 
let $e_n(G)$ be the number of elements of $G$ of order $n$. 
By Lemma \ref{PrimeDivLem}, it suffices for each pair of entries 
$(T,M_1) , (T,M_2)$ on the list from \cite{BurCov}, 
to exhibit a prime $p$ for which: 
\begin{equation} \label{MainIneq}
e_p(T) > e_p(M_1) + e_p(M_2). 
\end{equation}
Besides a family of pairs for which $T$ is alternating, 
which by Proposition 2.17 of \cite{VisVad} we can ignore, 
the list from \cite{BurCov} consists of (i) several infinite families of 
pairs $(T,M)$ for which both $T$ and $M$ are either 
symplectic or orthogonal groups, 
and (ii) a finite list of exceptional pairs, 
including several for which one of $T$ or $M$ is sporadic. 
Following some further preliminary observations in Section \ref{PrelimSect}, 
we prove the inequality (\ref{MainIneq}) in Section \ref{ClassicalSect} 
when both $(T,M_1) $ and $ (T,M_2)$ are of type (i) (with one exception), 
and in Section \ref{ExceptionSect} otherwise 
(the verification for the exceptional pairs 
being via computations in GAP\cite{GAP} and Magma\cite{Magma}). 
We conclude the article with a few remarks and questions 
concerning the extent to which a finite group is determined by 
its co-intersection graph (Section \ref{RecogSect}) 
and concerning co-intersection graphs of infinite groups (Section \ref{InfiniteSect}). 
Some highlights of this discussion are the following. 

\begin{propn}[Proposition \ref{RandomProp}]
An Erd\H{o}s-R\'{e}nyi random graph is not isomorphic to the 
co-intersection graph (or the intersection graph) 
of any finite group. 
\end{propn}

\begin{propn}[Proposition \ref{InfManyGenusProp}]
There are infinitely many positive integers $k$ such that there exists 
a finite graph $\Gamma_k$ for which there are exactly $k$ isomorphism classes 
of finite groups $G$ satisfying $\Delta_G ^c \cong \Gamma_k$ 
(indeed this is possible with $\Gamma_k$ a complete graph). 
\end{propn}

\begin{propn}[Proposition \ref{AbsRecogInftyProp}]
There are infinitely many isomorphism classes of finite groups $G$ for which, 
if $H$ is a finite group satisfying $\Delta_G ^c \cong \Delta_H ^c$, 
then $G \cong H$ (indeed this is possible with $G$ metacyclic). 
\end{propn}

\section{Preliminaries} \label{PrelimSect}

The following is a slight generalisation of the final part of 
Proposition 2.1 from \cite{VisVad}, with the same proof. 

\begin{lem} \label{Diam3Lemma}
For every finite group $G$, every connected component of $\Delta_G ^c$ has 
diameter at most three. 
\end{lem}

\begin{proof}
Suppose $H = H_0 \text{---} H_1 \text{---} \ldots \text{---} H_k = H'$ 
is a path of length $k$ in $\Delta_G ^c$ joining the vertices $H$ and $H'$. 
Let $M_i \leq H_i$ be a minimal subgroup. 
Then $H \text{---} M_1 \text{---} M_{k-1} \text{---} H'$ 
is a path of length at most $3$ 
in $\Delta_G ^c$ joining $H$ and $H'$. 
\end{proof}

\begin{propn} \label{CompProp}
The graph $\Delta_G ^c$ is nonempty complete iff there exist primes 
$p$ and $q$ (not necessarily distinct) such that $\lvert G \rvert = pq$. 
\end{propn}

\begin{proof}
If $G$ has a unique nontrivial proper subgroup, 
then $G$ is cyclic of order $p^2$. 
Otherwise, the conclusion is immediate from 
Lemma 2.5 and Proposition 2.6 of \cite{VisVad}. 
\end{proof}

\begin{thm}[\cite{BurCov} Theorem 2.2]
Let $T$ be a finite simple group and let $M$ be a maximal subgroup of $M$. 
Then $\lvert T \rvert$ and $\lvert M \rvert$ have the same set of prime 
divisors iff $(T,M)$ is as in Table \ref{tab:BadPairs}. 
\end{thm}

\begin{table}
    \centering
    \begin{tabular}{llll}
         & $T$ & \textbf{Type of} $M$ & \textbf{Conditions} \\
        (a) & $A_n$ & $(S_k \times S_{n-k}) \cap A_n$ & No primes in $[k+1,n]\cap \mathbb{N}$ \\
        (b) & $\Sp_{2m} (q)$ & $O^- _{2m} (q)$ & $m,q$ even \\
        (c) & $\Omega_{2m+1} (q)$ & $O^- _{2m} (q)$ & $m$ even, $q$ odd \\
        (d) & $\text{P} \Omega_{2m} ^+ (q)$ & $O _{2m-1} (q)$ & $m$ even, $q$ odd \\
        (e) & $\text{P} \Omega_{2m} ^+ (q)$ & $\Sp_{2m-2} (q)$ & $m,q$ even \\
        (f) & $\PSp_4 (q)$ & $\Sp_2 (q^2)$ & \\
        & $L_6(2)$  & $P_1 , P_5$ & \\
        & $U_3(3)$ & $L_2(7)$ &   \\
        & $U_3(5)$ & $A_7$ &   \\
        & $U_4(2)$ & $P_2 , \Sp_4 (2)$ &   \\
        & $U_4(3)$ & $L_3(4) , A_7$ &   \\
        & $U_5(2)$ & $L_2(11)$ &   \\
        & $U_6(2)$ & $M_{22}$ &   \\
        & $\PSp_4 (7)$ & $A_7$ &   \\
        & $\Sp_6 (2)$ & $O_6 ^+ (2)$ &   \\
        & $\Omega_8 ^+ (2)$ & $P_1,P_3,P_4,A_9$ &   \\
        & $G_2(3)$ & $L_2(13)$ &   \\
        & ${^2}F_4(2)'$ & $L_2(25)$ &   \\
        & $M_{11}$ & $L_2(11)$ &   \\
        & $M_{12}$ & $M_{11} , L_2(11)$ &   \\
        & $M_{24}$ & $M_{23}$ &   \\
        & $\text{HS}$ & $M_{22}$ &   \\
        & $\text{McL}$ & $M_{22}$ &   \\
        & $\Co_2$ & $M_{23}$ &   \\
        & $\Co_3$ & $M_{23}$ &   \\
    \end{tabular}
    \caption{Pairs $(T,M)$ with $T$ nonabelian finite simple; 
    $M <_{\max} T$, and $\lvert T \rvert$, $\lvert M \rvert$ having the same 
    sets of prime divisors. }
    \label{tab:BadPairs}
\end{table}

\section{Symplectic and orthogonal groups} \label{ClassicalSect}

We refer to Subsection 1.6.4 of \cite{BrHoRD} for
the orders of the finite simple orthogonal and symplectic groups
and their associated covers and overgroups.

\begin{rmrk} \label{SmallRkIsoRmrk}
In studying the simple groups $T = \text{P} \Omega^{\epsilon} _n (q)$
and $T = \PSp_n(q)$ we shall assume
$n \geq 7$ and $n \geq 4$, respectively,
since for smaller values of $n$ the corresponding groups
either fail to be simple, or are isomorphic to simple groups
from other families (see Proposition 1.10.1 of \cite{BrHoRD}).
\end{rmrk}

It should be noted that the data in Table \ref{tab:BadPairs} does not 
give us the exact shape of the subgroup $M <_{\max} T$ in the cases 
(b)-(f). The tables in Section 2.2 of \cite{BrHoRD} detail the exact 
shape of the corresponding subgroup in the (generally quasisimple) 
cover of $T$, from which the following are immediate. 

\begin{lem}
    Let $(T,M)$ be one of the pairs in lines (b)-(f) of table \ref{tab:BadPairs}. 
    Then $M$ is a central quotient of the following group $\tilde{M}$. 
    \begin{itemize}
        \item[(b)] $\tilde{M} = O^- _{2m} (q)$; 
        \item[(c)] $\tilde{M} = \Omega^- _{2m} (q).2$; 
        \item[(d)] $\tilde{M} = \Omega_{2m-1} (q).2$; 
        \item[(e)] $\tilde{M} = \Sp_{2m-2} (q)$; 
        \item[(f)] $\tilde{M} = \Sp_2 (q^2).2$.
    \end{itemize}
\end{lem}

For $q$ a prime power and $e \geq 2$, a prime divisor $r$ of $q^e - 1$
is \emph{primitive} if $r$ does not divide $q^i - 1$ for $1 \leq i \leq e-1$.

\begin{thm}[Zsigmondy]
For all $e \geq 3$ and all prime powers $q$ with $(q,e) \neq (2,6)$,
$q^e - 1$ has a primitive prime divisor.
\end{thm}

We follow the description of the conjugacy classes and centralizers
of elements of prime order in finite simple symplectic and orthogonal
groups $G = \PSp_n(q)$ or $\text{P} \Omega_n ^{\epsilon} (q)$ from Subsection 2.2 of \cite{BurCov}.
To wit, for each prime divisor $r$ of $\lvert G \rvert$
such that $r \nmid q$, let $i \in \mathbb{N}$ be minimal
such that $r \mid (q^i - 1)$.
Then the conjugacy classes of elements of order $r$ in $G$
are parametrized by tuples $a_1 , \ldots , a_t$ of positive integers
satisfying $i \leq i (a_1 + \ldots a_t) \leq n$
(and additional conditions in some cases).
The only case that shall concern us is the following:

\begin{lem} \label{CentralizerOrders}
If $i > n/2$, then $G$ contains a unique conjugacy class of $r$-elements,
corresponding to the data $t=1$ and $(a_1) = (1)$.
In this case,
\begin{equation*}
\lvert C_G (x) \rvert
= \left\{ \begin{array}{ll}
 2^{-a} \lvert \Sp_l (q) \rvert  \lvert \GU_1 (q^{i/2}) \rvert & G = \PSp_n(q), i \text{ even} \\
 2^{-a} \lvert \Sp_l (q) \rvert \lvert \GL_1 (q^i) \rvert & G = \PSp_n(q), i \text{ odd} \\
 2^{-a} \lvert \SO_l ^{\epsilon'} (q) \rvert \lvert \GU_1 (q^{i/2}) \rvert & G = \text{P} \Omega^{\epsilon} (q), i \text{ even} \\
 2^{-a} \lvert \SO_l ^{\epsilon'} (q) \rvert \lvert \GL_1 (q^i) \rvert & G = \text{P} \Omega^{\epsilon} (q), i \text{ odd}
\end{array} \right.
\end{equation*}
where $x \in G$ is an element of order $r$, $l=n-i$ and
$a = (1-(-1)^q)/2$.
\end{lem}

Note that $\lvert \GL_1 (q^i) \rvert = q^i - 1$
and $\lvert \GU_1 (q^{i/2}) \rvert = q^{i/2} + 1$.
In fact, in all cases relevant to us, 
we shall have either that $l=1$ (only for $G$ orthogonal)
or $l=2$.
Note that $\SO_1 ^{\epsilon'} (q)$ is trivial;
$\SO_2 ^{\epsilon'} (q) \cong C_{q-\epsilon' }$
and $\Sp_2 (q) \cong \SL_2 (q)$.

\begin{thm} \label{GenericPairsThm}
For all pairs $(T,M)$ in rows (b)-(f) of Table \ref{tab:BadPairs},
there exists a prime $r$ such that $r \mid \lvert T \rvert$ and
$e_r (T) > 2 e_r (M)$. (with the possible exception of the pair
$(T,M) = (\Omega_8 ^+ (2), \Sp_6 (2))$ in row (e)).
\end{thm}

\begin{proof}
In each case, we identify an odd prime $r$ which is a primitive prime divisor
or $q^i - 1$ for $i$ even and sufficiently large that
there is a unique conjugacy class of $r$-elements in both $T$ and $M$.
We then estimate:
\begin{equation*}
    \frac{e_r(T)}{e_r(M)} = \frac{\lvert \ccl_T (x) \rvert}{\lvert \ccl_M (x) \rvert}
  = \frac{\lvert T \rvert}{\lvert M \rvert} \cdot \frac{\lvert C_M (x) \rvert}{\lvert C_T (x) \rvert}
\end{equation*}
for any $r$-element $x \in M$.

For case (b), since $m \geq 2$ is even, there is a primitive prime divisor $r \mid (q^{2m}-1)$, so $i=2m$.
Then by Lemma \ref{CentralizerOrders},
$\lvert C_M (x) \rvert = \lvert C_T (x) \rvert = \lvert \GU_1 (q^m) \rvert$, so we have:
$$\frac{e_r(T)}{e_r(M)} \geq \frac{\lvert T \rvert}{\lvert M \rvert} = q^m (q^m - 1) > 2. $$
For case (c), there is again a primitive prime divisor $r \mid (q^{2m}-1)$, so $i=2m$.
By Lemma \ref{CentralizerOrders} we have:
$$\lvert C_T(x) \rvert \leq \frac{\lvert \SO_1 ^{\epsilon'} (q) \rvert}{2} \lvert \GU_1 (q^m) \rvert \leq \lvert \GU_1 (q^m) \rvert$$
and:
$$\lvert C_T(x) \rvert \geq \lvert C_M(x) \rvert \geq \lvert C_{\Omega_{2m} ^- (q)}(x) \rvert \geq \lvert \GU_1 (q^m) \rvert \leq \lvert C_T(x) \rvert$$
so:
$$\frac{e_r(T)}{e_r(M)} \geq \frac{\lvert T \rvert}{\lvert M \rvert}
\geq \frac{1}{2} q^m (q^m - 1) > 2. $$

For case (d), we take instead a primitive prime divisor $r$ of
$q^{2(m-1)} - 1$, so that $i=2m-2$.
For $T$, we are applying Lemma \ref{CentralizerOrders}
with $l=2$, so that:
$$\lvert C_T(x) \rvert = \frac{1}{2} \lvert \SO_2 ^{\epsilon '} (q) \rvert \lvert \GU_1 (q^{m-1}) \rvert$$
while for $M$ we apply Lemma \ref{CentralizerOrders} with $l=1$, so:
$$\lvert C_M (x) \rvert \geq \lvert C_{\Omega_{2m-1}(q)} (x) \rvert \geq \lvert \GU_1 (q^{m-1}) \rvert/2$$
thus:
$$\frac{e_r(T)}{e_r(M)} \geq \frac{q^{m-1} (q^m - 1)}{4 (q+1) } > 2$$
since by Remark \ref{SmallRkIsoRmrk} we may take $m \geq 4$.

Similarly for case (e), we take a primitive prime divisor $r$ of
$q^{2(m-1)} - 1$, so that $i=2m-2$, and apply
Lemma \ref{CentralizerOrders} with $l=2$ for $T$
and $l=0$ for $M$, so that:
$$\lvert C_T(x) \rvert = \lvert \SO_2 ^{\epsilon '} (q) \rvert \lvert \GU_1 (q^{m-1}) \rvert$$
and:
$$\lvert C_M(x) \rvert = \lvert \GU_1 (q^{m-1}) \rvert$$
so that:
$$\frac{e_r(T)}{e_r(M)} = \frac{1}{\lvert \SO_2 ^{\epsilon'} (q) \rvert}\frac{\lvert T \rvert}{\lvert M \rvert} \geq \frac{q^{m-1}(q^m - 1)}{q+1} > 2$$
since again $m \geq 4$.
Note that we must exclude $\Omega_4 ^+ (2)$ from this
analysis, since for the pair $(q,m)=(2,4)$,
$e=6$ fails the conclusion of Zsigmondy's Theorem.

Finally for case (f), let $r$ be a primitive prime divisor of $q^4 - 1$, then much as in case (b),
$C_M(x) = C_T(x)$, so:
$$\frac{e_r(T)}{e_r(M)} \geq \frac{\lvert T \rvert}{\lvert M \rvert}
\geq \frac{q^2 (q^2 - 1)}{2 (2,q-1)} > 2$$
as desired.
\end{proof}

\begin{rmrk} \label{OverlapRmrk}
\normalfont
\begin{itemize}
\item[(i)] From the proof of Theorem \ref{GenericPairsThm}, 
we have that for the pairs $(T,M)$ 
in rows (b) or (f), respectively, of Table \ref{tab:BadPairs}, 
any primitive prime divisor $r$ of $q^{2m} - 1$ or $q^4 - 1$, 
respectively, satisfies the conclusion of the Theorem. 

\item[(ii)] In particular, letting $T = \Sp_4 (q)$ for $q$ even, 
if $r$ is a primitive prime divisor of $q^4 - 1$, 
and $M_1 , M_2 < T$ are maximal subgroups of type (b) and (f), 
respectively, then: 
\begin{equation*}
e_r (T) > e_r(M_1) + e_r(M_2). 
\end{equation*}
\item[(iii)] The only groups $T$ which feature simultaneously
in rows (b)-(f) of Table \ref{tab:BadPairs} and in
the outstanding cases appearing below them are
$T = \PSp_4 (7)$ (case (f)) and $T =  \Omega ^+ _8 (2)$ (case (e)). 

\end{itemize}
\end{rmrk}

\section{Outstanding cases} 
\label{ExceptionSect}

As before, it is sufficient to verify that for all pairs $(T,M)$ in Table \ref{tab:BadPairs}, there exists a prime $p$ such that $e_p(T) >2e_p(M),$ which immediately implies that $T$ has Property (P). In these cases we observe something stronger, namely that  $$e_p(T) >2\vert M\vert.$$ We verify these using computations in GAP\cite{GAP} and Magma\cite{Magma}. In Table \ref{tab:outstanding}, for each $(T,M)$, we record: \begin{itemize} \item the relevant prime $p$, \item the number of elements of order $p$ in $T$ \item  $e_p(M),$ and \item $|M|.$ \end{itemize}

\begin{table}[ht]
\centering
\renewcommand{\arraystretch}{1.2}
\begin{tabular}{|c|c|c|c|c|c|}
\hline
$T$ & $M$ & prime $p$ & $e_p (T)$  & $e_p(M)$ & $\lvert M \rvert$ \\ \hline
$L_6(2)$ & $P_1,\,P_5$ & $p=31$:& $3\,901\,685\,760$ \;& $61\,931\,520$ & $319\,979\,520$ \\ \hline
$U_3(3)$ & $L_2(7)$ & $p=7$&$1\,728$  &  $48$&$168$ \\ \hline
$U_3(5)$ & $A_7$ & $p=7$ & $36\,000$ & $720$&$2\,520$ \\ \hline
$U_4(2)$ 
& \makecell{$P_2$ \\ $Sp_4(2)$}
& $p=5$
& $5\,184$
& \makecell{$384$ \\ $144$}
& \makecell{$960$ \\ $720$}
\\ \hline
$U_4(3)$ 
& \makecell{$L_3(4)$ \\ $A_7$}
& $p=7$
& $933\,120$
& \makecell{$5\,760$ \\ $720$}
& \makecell{$20\,160$\\ $2\,520$ }
\\ \hline

$U_5(2)$ & $L_2(11)$ & $p=11$&$2\,488\,320$ & $120$&$660$\\ \hline
$U_6(2)$ & $M_{22}$ & $p=11$&$1\,672\,151\,040$& $80\,640$&$443\,520$ \\ \hline
$PSp_4(7)$ & \makecell{$A_7$\\$PSp_2(49):2$} & $p=2$& $52\,675$ &\makecell{$350$ \\ $2\,450$}&\makecell{$2\,520$ \\ $117\,600$} \\ \hline
$Sp_6(2)$ & $O_6^+(2)$ & $p=7$& $207\,360$ &$5\,760$&$40\,320$ \\ \hline
$\Omega_8^+(2)$ & \makecell{$P_1,P_3,P_4$\\$A_9$\\$\Sp_6(2)$} & $p=7$& $24\,883\,200$ & \makecell{$368\,640$ \\ $25\,920$ \\ $207\,360$ }&\makecell{$1\,290\,240$ \\ $181\,440$ \\ $1\,451\,520$ }\\ \hline
$G_2(3)$ & $L_2(13)$ & $p=13$& $653\,184$& $168$&$1\,092$\\ \hline
${}^2F_4(2)'$ & $L_2(25)$ & $p=13$& $2\,764\,800$ &$3\,600$&$7\,800$ \\ \hline
$M_{11}$ & $L_2(11)$ & $p=11$& $1\,440$ &$660$& $120$ \\ \hline
$M_{12}$ & \makecell{$M_{11}$\\$L_2(11)$} & $p=11$& $17\,280$ & \makecell{$1\,440$ \\ $120$} &\makecell{$7\,920$\\$660$}\\ \hline
$M_{24}$ & $M_{23}$ & $p=23$& $21\,288\,960$ & $887\,040$ &$10\,200\,960$\\ \hline
$HS$ & $M_{22}$ & $p=11$& $8\,064\,000$& $80\,640$& $443\,520$ \\ \hline
$McL$ & $M_{22}$ & $p=11$& $163\,296\,000$& $80\,640$& $443\,520$ \\ \hline
$Co_2$ & $M_{23}$ & $p=23$& $3\,678\,732\,288\,000$ &$887\,040$ &$10\,200\,960$ \\ \hline
$Co_3$ & $M_{23}$ & $p=23$& $43\,110\,144\,000$ &$887\,040$ &$10\,200\,960$ \\ \hline
\end{tabular}
\caption{Outstanding cases}
\label{tab:outstanding}
\end{table}

\section{Recognizability by (co-)intersection graph} \label{RecogSect}

One may ask to what extent the finite group $G$ is determined by $\Delta_G ^c$; 
for which finite graphs $\Gamma$ there exists $G$ for which $\Gamma \cong \Delta_G ^c$, 
and when such $G$ exists, how many such $G$ exist (up to isomorphism), 
and what structural features they must share. 

\begin{defn}
Let $\Gamma$ be a finite graph.
The \emph{co-intersection genus} $\coIG(\Gamma)$ of $\Gamma$
is the set of all (isomorphism classes of) finite groups
$G$ satisfying $\Delta_G ^c \cong \Gamma$.
For $\mathcal{C}$ a class of finite groups,
and $G \in \mathcal{C}$ we say that $G$ is
\emph{recognizable in $\mathcal{C}$} by its co-intersection graph if
$\coIG(\Delta^c _G) \cap \mathcal{C} = \lbrace G \rbrace$.
We say that $G$ is
\emph{absolutely recognizable} by its co-intersection graph
if it is recognizable by its co-intersection graph in
the class of all finite groups.
\end{defn}

\begin{rmrk}
    \begin{itemize}
        \item[(i)] We concurrently have a notion of \emph{intersection genus} $\IG(\Gamma)$ 
        of the graph $\Gamma$, so that $\coIG(\Gamma) = \IG(\Gamma^c)$. 
        In particular, the group $G$ is recognizable by its intersection 
        graph (in a class $\mathcal{C}$) iff $G$ is recognizable by its co-intersection 
        graph. 
        
        \item[(ii)] The subgroup lattice $L(G)$ of the group $G$ determines $\Delta_G$, 
        hence the (co-)intersection genus of $G$ contains all groups $H$ for 
        which $L(G) \cong L(H)$. 
        
    \end{itemize}
\end{rmrk}

For many finite graphs $\Gamma$, $\IG(\Gamma)$ and $\coIG(\Gamma)$ are empty. 

\begin{ex}
By Lemma \ref{Diam3Lemma}, if $\Gamma$ has any connected component of diameter at least four, 
then there is no finite group $G$ for which $\Delta_G ^c \cong \Gamma$. 
\end{ex}

In fact, in an appropriate sense, for ``most'' finite graphs, 
$\IG(\Gamma)$ and $\coIG(\Gamma)$ are empty. 

\begin{propn} \label{RandomProp}
Let $\Gamma_n$ be an Erd\H{o}s-R\'{e}nyi random graph with $n$ vertices. 
Then as $n \rightarrow \infty$, 
the probability that $\Gamma_n$ is isomorphic to the intersection or co-intersection graph of 
any finite group tends to $0$. 
\end{propn}

\begin{rmrk}
Recall that an Erd\H{o}s-R\'{e}nyi random graph is connected of diameter two with high 
probability. In other words, Proposition \ref{RandomProp} 
for co-intersection graphs cannot be deduced 
from Lemma \ref{Diam3Lemma}. 
\end{rmrk}

Recall that with probability $\rightarrow 1$ as $n \rightarrow \infty$, 
    the random graph $\Gamma_n$ admits no nontrivial graph automorphism 
    (see \cite{ErdRen}). Thus Proposition \ref{RandomProp} follows immediately from 
    the next Proposition. 

\begin{propn} \label{AutomorphProp}
Let $G$ be a finite group with at least two subgroups. 
Then $\Delta _G$ and $\Delta^c _G$ both have a nontrivial graph automorphism. 
\end{propn}

The proof of Proposition \ref{AutomorphProp} rests on the following Lemma, 
which must be well-known, but being unable to locate a published reference 
for the proof we include one here. 

\begin{lem} \label{CharCycLem}
    Let $G$ be a finite group. Suppose that every subgroup of $G$ is characteristic. 
    Then $G$ is cyclic. 
\end{lem}

\begin{proof}
First, such $G$ must be abelian. 
For every subgroup of $G$ is normal, 
so by Dedekind's classification of Hamiltonian groups \cite{Dedek}, 
if $G$ is nonabelian then $G$ has $Q_8$ as a direct factor. 
But $Q_8$ admits an automorphism which cyclically permutes its three $C_4$-subgroups. 

Similarly, by the primary decomposition of $G$, we reduce to the case for which $G$ 
is an abelian $p$-group. 
If $G$ is not cyclic, then $G$ has a direct factor of the form 
$(\mathbb{Z}/p^k \mathbb{Z}) \times (\mathbb{Z}/p^l \mathbb{Z})$ for some $1 \leq k \leq l$, 
in which the subgroup $\lbrace 0 \rbrace \times (\mathbb{Z}/p^l \mathbb{Z})$ 
is not preserved by the automorphism: 
$$(a + p^k \mathbb{Z},b + p^l \mathbb{Z}) \mapsto (a + b + p^k \mathbb{Z},b + p^l \mathbb{Z}).$$
\end{proof}

\begin{proof}[Proof of Proposition \ref{AutomorphProp}]
Since every automorphism of the graph $\Gamma$ induces an automorphism of $\Gamma^c$, 
we need only prove the conclusion for $\Delta^c _G$. 
First, any group automorphism of $G$ induces a graph automorphism of $\Delta^c _G$. 
If a nontrivial graph automorphism is induced as such, we are done. 
Otherwise Lemma \ref{CharCycLem} applies and $G$ is cyclic. 

In this case, let the prime factorisation of $\lvert G \rvert$ be $p_1 ^{a_1} \cdots p_m ^{a_m}$. 
Recall that for each divisor $d$ of $\lvert G \rvert$, 
there is a unique subgroup of $G$ of order $d$. 
If there exist $1 \leq i < j \leq m$ such that $a_i = a_j$ then 
swapping $p_i$ and $p_j$ induces a nontrivial permutation of the subgroups of $G$ 
which preserves adjacency in $\Delta^c _G$. 
If there exists $i$ for which $a_i \geq 3$ then 
the subgroups of $G$ of order $\lvert G \rvert/p_i$ and $\lvert G \rvert/p_i ^2$ 
are isolated points of $\Delta^c _G$, hence there is a nontrivial 
graph automorphism swapping them. 
If neither possibility pertains, 
then since $G$ has at least two subgroups we have $\lvert G \rvert = p_1 p_2 ^2$, 
and there is an automorphism of $\Delta^c _G$ swapping the subgroups of orders $p_2$ and $p_2 ^2$. 
\end{proof}

For many other finite graphs, $\coIG(\Gamma)$ is infinite. 

\begin{ex}
\item[(i)] If $\Gamma$ is the empty graph of order $n$,
then $\coIG(\Gamma)$ consists of the cyclic groups $C_{p^{n+1}}$
(for all primes $p$) and, in the case $n=2^{m-1} + m - 2$,
the generalized quaternion group $Q_{2^m}$
(see Theorem 1 of \cite{DevRaj}).
In particular, $\coIG(\Gamma)$ is infinite. 

\item[(ii)] If $G$ is a finite cyclic group, and $\lvert G \rvert $ 
has prime factorisation $p_1 ^{a_1} \cdots p_m ^{a_m}$, 
then for any other tuple $q_1 , \cdots , q_m$ of distinct prime numbers, 
the cyclic group $H$ of order $q_1 ^{a_1} \cdots q_m ^{a_m}$ 
has a subgroup lattice isomorphic to that of $G$, 
hence $\Delta_G ^c \cong \Delta_H ^c$. 
\end{ex}

Finally, one may ask which finite graphs have finite nonempty 
co-intersection genus and moreover, for which $n \in \mathbb{N}$ 
there exists a finite graph $\Gamma$ for which $\lvert \coIG(\Gamma) \rvert = n$. 

\begin{ex} \label{CoIntFinGenEx}
If $\Gamma = K_n$ is a finite complete graph,
then any group in $\coIG(\Gamma)$
is of the form $C_p \times C_p$;
$C_p \times C_q$ or $C_p \rtimes C_q$ for $q \mid (p-1)$
(see Lemma 2.5 and Proposition 2.6 of \cite{VisVad}).
In particular, $\coIG(K_2)$ is infinite;
$\coIG(K_n)$ is empty for $n-1 \geq 2$ composite,
and for $p=n-1 \geq 2$ prime, $\lvert \coIG(K_n) \rvert -1=\omega(p-1)$
is the number of distinct prime divisors of $p-1$.
Thus, the Klein four-group is
absolutely recognizable by its co-intersection graph. 
Further, it is known that for all $1 \leq k \leq 100$, 
there exists a prime $p$ such that $\omega(p-1)=k$ (even with $p-1$ squarefree) \cite{OEIS}. 
For such $p$, $\lvert \coIG(K_{p+1}) \rvert = k+1$. 
\end{ex}

Thus understanding the possible co-intersection genera of 
complete graphs is equivalent to the following Number Theory problem. 

\begin{conj} \label{PrimesConj}
Let $\mathbb{P}$ be the set of prime numbers and let $f : \mathbb{P} \rightarrow \mathbb{N}$ 
be given by $f(p) = \omega(p-1)$. Then $f$ is surjective. 
\end{conj}

Expert opinion appears to be that Conjecture \ref{PrimesConj} is very likely to be true, 
and ``follows in much stronger form from standard (difficult) conjectures'' \cite{Kowa}. 

\begin{propn} \label{InfManyGenusProp}
The function $f$ from Conjecture \ref{PrimesConj} has infinite image. 
Thus, there are infinitely many integers $k$ such that there exists 
a prime $p$ for which $\lvert \coIG(K_{p+1}) \rvert = k$. 
\end{propn}

\begin{proof}
The second statement is immediate from the first and Example \ref{CoIntFinGenEx}. 
Let $q_1 , \ldots , q_n$ be any distinct prime numbers and let $Q = q_1 \cdots q_n$. 
By Dirichlet's Theorem, there exists a prime $p$ satisfying $p \equiv 1 \mod Q$. 
Thus $\omega(p-1) \geq n$. Since $n$ was arbitrary, the result follows.  
\end{proof}

Our next result is based on the tools for recognition by 
(co-)intersection graph developed in \cite{ShaKohNonSimple}. 

\begin{propn} \label{AbsRecogInftyProp}
There are infinitely many isomorphism classes of metacyclic finite groups $G$ 
such that $G$ is absolutely recognizable by (co-)intersection graph. 
\end{propn}

Let $\omega(\Gamma)$ be the clique number of the graph $\Gamma$.
We note that for $G$ a finite group, $\omega (\Delta^c _G)$
is precisely the number of minimal subgroups of $G$
(see \cite{VisVad} Proposition 3.1). 

\begin{proof}[Proof of Proposition \ref{AbsRecogInftyProp}]
    In Theorem 3.1 of \cite{ShaKohNonSimple} the finite nonsimple groups $G$ for which 
    $\Delta_G$ has a unique leaf are classified. 
    They are precisely groups of the form $C_p \rtimes C_{q^2}$ 
    (see Table 1 of \cite{ShaKohNonSimple}). 
    Among these, it is easy to see that those with a cyclic subgroup 
    of order $pq$ are distinguished from those without:
    for groups $G$ of the former type 
    we have $\omega (\Delta^c _G) = 2$, 
    whereas for the latter type 
    we have $\omega (\Delta^c _G) = p+1$. 
    For the same reason, among groups of the latter type $\Delta^c _G$ 
    detects the value of $p$. 

    Therefore, let $G_{p,q} = C_p \rtimes C_{q^2}$, 
    with the generator for $C_{q^2}$ acting as an automorphism of $C_p$ of order $q^2$ 
    (so that $q^2 \mid (p-1)$). 
    By the above, if $H$ is a finite group satisfying $\Delta^c _{G_{p,q}} \cong \Delta^c _H$, 
    then $H = G_{p,r}$ for $r$ a prime satisfying $r^2 \mid (p-1)$. 
    It therefore suffices to prove there are infinitely many primes $p$ for which 
    exactly one prime $q$ satisfies $q^2 \mid (p-1)$. 
    We shall achieve this with $q=2$. 
    First, by Dirichlet's Theorem, the set of primes $p \equiv 1 \mod 4$ 
    has relative density $1/2$ in the set of all primes. 
    Second, by Theorem 6 of \cite{BrouZhou}, 
    the set of primes of the form $p= 2^e \overline{p}$, for $\overline{p}$ 
    an odd squarefree number, has relative density $> 0.74$. 
    Thus the set of primes $p$ satisfying our condition has relative density 
    $> 0.24$, hence is infinite. 
\end{proof}

Turning to simple groups, it is known 
that if $T$ and $G$ are finite groups with isomorphic subgroup lattices 
and $T$ is nonabelian simple, then $G \cong T$ (see Theorem 7.8.1 of \cite{Schmidt}). 
Although the (co-)intersection graph retains much less information than the 
subgroup lattice, at present we do not have a counterexample to the following. 

\begin{qu} \label{SimpleRecogQ}
Is every nonabelian finite simple group absolutely recognizable by its 
(co-)intersection graph? 
\end{qu}

\begin{rmrk}
\normalfont
In \cite{ShaKohSimple} various families of finite simple groups are described for which 
Question \ref{SimpleRecogQ} has a positive answer (see Table 3 therein). 
For instance, if $p$ is a Sophie Germain prime, then $A_{2p+1}$ 
is absolutely recognizable by its (co-)intersection graph. 
It appears to be unknown whether there are infinitely many nonabelian finite simple groups which 
are absolutely recognizable by (co-)intersection graph: 
for instance it remains unknown whether there are infinitely many Sophie Germain primes, 
and there are similar number-theoretic uncertainties for each of the other 
items in Table 3 of \cite{ShaKohSimple}. 
\end{rmrk}

It seems likely that the two integers $\lvert V(\Delta^c _T) \rvert$ 
(the number of proper nontrivial subgroups of $T$) 
and $\omega (\Delta^c _T)$ (the number of minimal subgroups of $T$) 
is sufficient information to identify 
$T$ among all nonabelian finite simple groups. 
We therefore propose: 

\begin{conj} \label{SimpleRecogConj}
Let $\mathcal{S}$ be the class of nonabelian finite simple groups.
Then every $T \in \mathcal{S}$
is recognizable in $\mathcal{S}$ by its (co-)intersection graph.
\end{conj}

\begin{rmrk}
\normalfont
In support of Conjecture \ref{SimpleRecogConj}, 
using GAP we ascertained that there are no $T_1 , T_2 \in \mathcal{S}$ 
of order less than one million, for which $T_1 \ncong T_2$ 
but $\omega (\Delta^c _{T_1}) = \omega (\Delta^c _{T_2})$. 
In addition we computed the total number of subgroups 
for each group $T \in \mathcal{S}$ of order less than one million 
and found only one pair with the same total, 
namely $\PSL_2 (53)$ and $\PSL_2 (59)$. 
\end{rmrk}

In the meantime, we have the following easy observation, 
which allows recognition within any nested sequence 
of finite simple groups. 

\begin{lem}
Let $G$ be a nonabelian finite simple group and let $1 \neq H \lneq T$. Then $\omega(\Delta^c _T) > \omega(\Delta^c _H)$.
\end{lem}
\begin{proof}
    There is at least one element of prime order in $T\setminus H$.
\end{proof}

\begin{ex}
For every finite graph $\Gamma$,
$\coIG(\Gamma)$ contains at most one finite alternating group of degree $\geq 5$.
\end{ex}

\section{Infinite groups} \label{InfiniteSect}

One may also define the co-intersection graph $\Delta^c _G$ 
for $G$ an infinite group. 
We recall that every infinite group has infinitely many subgroups: 
if $G$ contains an element of infinite order, 
then this follows from the fact that $\mathbb{Z}$ has infinitely many subgroups, 
whereas if $G$ is an infinite torsion group then it is enough to recall that 
$G$ is the union of its cyclic subgroups. 
Hence $\Delta^c _G$ is always an infinite graph. 

\begin{ex}
\normalfont
If $G = \mathbb{Z}$ or $G = C_{p^{\infty}}$ (the Pr\"{u}fer-$p$-group, for $p$ a prime number) 
then $\Delta^c _G$ is an infinite empty graph. 
By contrast, if there exists a prime number $p$ 
such that every proper nontrivial subgroup of $G$ is isomorphic to $C_p$, 
then $\Delta^c _G$ is an infinite complete graph. 
Groups of the latter type exist for all sufficiently large prime 
numbers $p$; they are known as \emph{Tarski monsters} \cite{Olsh}. 
Consequently, the countably infinite empty graph 
and the countably infinite complete graph have infinite co-intersection genus. 
\end{ex}

At present we know very little about the general shape of co-intersection 
graphs of infinite groups, such that even the following remains unknown. 

\begin{qu} \label{InfiniteUnivQu}
Let $\Gamma$ be an infinite graph. 
Does there exist an infinite group $G$ such that $\Gamma \cong \Delta^c _G$? 
\end{qu}

Note that the analogue of Question \ref{InfiniteUnivQu}, for which we 
replace the word ``infinite'' by ``finite'', 
has a negative answer thanks to Lemma \ref{Diam3Lemma}. 
The proof of Lemma \ref{Diam3Lemma} cannot be extended to the setting of 
infinite groups, since a proper subgroup of an infinite group need not 
contain a minimal subgroup. 
Therefore we have the following special case of Question \ref{InfiniteUnivQu}, 
which was posed as a question to us by I. Wong. 

\begin{qu}
    Does there exist an infinite group $G$ such 
    that $\Delta^c _G$ has a connected component of diameter at least four? 
    What about a connected component of infinite diameter? 
\end{qu}

\subsection*{Acknowledgments} This project began at the conference 
``Topics in Group Theory - on the occasion of Andrea Lucchini's 60th(+) birthday'', 
held at the University of Padova in September 2024. We wish to record our gratitude to 
the conference organisers for providing an inspiring environment in which to work.

\end{document}